\documentclass[
11pt]{amsart}

\usepackage{amsmath}
\usepackage{amssymb}
\usepackage{latexsym}
\usepackage{amsfonts}
\usepackage{graphpap}

\begin{document}

\newcommand{\ci}[1]{_{ {}_{\scriptstyle #1}}}
\newcommand{\ti}[1]{_{\scriptstyle \text{\rm #1}}}

\newcommand{\norm}[1]{\ensuremath{\|#1\|}}
\newcommand{\abs}[1]{\ensuremath{\vert#1\vert}}
\newcommand{\nm}{\,\rule[-.6ex]{.13em}{2.3ex}\,}

\newcommand{\lnm}{\left\bracevert}
\newcommand{\rnm}{\right\bracevert}

\newcommand{\p}{\ensuremath{\partial}}
\newcommand{\pr}{\mathcal{P}}

\newcounter{vremennyj}

\newcommand\cond[1]{\setcounter{vremennyj}{\theenumi}\setcounter{enumi}{#1}\labelenumi\setcounter{enumi}{\thevremennyj}}

\newcommand{\pbar}{\ensuremath{\bar{\partial}}}
\newcommand{\db}{\overline\partial}
\newcommand{\D}{\mathbb{D}}
\newcommand{\T}{\mathbb{T}}
\newcommand{\C}{\mathbb{C}}
\newcommand{\N}{\mathbb{N}}
\newcommand{\bP}{\mathbb{P}}

\newcommand{\bS}{\mathbf{S}}
\newcommand{\bk}{\mathbf{k}}

\newcommand\cE{\mathcal{E}}
\newcommand\cP{\mathcal{P}}
\newcommand\cC{\mathcal{C}}
\newcommand\cH{\mathcal{H}}
\newcommand\cU{\mathcal{U}}
\newcommand\cQ{\mathcal{Q}}

\newcommand{\be}{\mathbf{e}}

\newcommand{\la}{\lambda}
\newcommand{\e}{\varepsilon}

\newcommand{\td}{\widetilde\Delta}

\newcommand{\tto}{\!\!\to\!}
\newcommand{\wt}{\widetilde}
\newcommand{\shto}{\raisebox{.3ex}{$\scriptscriptstyle\rightarrow$}\!}

\newcommand{\La}{\langle }
\newcommand{\Ra}{\rangle }
\newcommand{\ran}{\operatorname{ran}}
\newcommand{\tr}{\operatorname{tr}}
\newcommand{\codim}{\operatorname{codim}}
\newcommand\clos{\operatorname{clos}}
\newcommand{\spn}{\operatorname{span}}
\newcommand{\lin}{\operatorname{Lin}}
\newcommand{\rank}{\operatorname{rank}}
\newcommand{\re}{\operatorname{Re}}
\newcommand{\vf}{\varphi}
\newcommand{\f}{\varphi}


\newcommand{\entrylabel}[1]{\mbox{#1}\hfill}

\newenvironment{entry}
{\begin{list}{X}%
  {\renewcommand{\makelabel}{\entrylabel}%
      \setlength{\labelwidth}{55pt}%
      \setlength{\leftmargin}{\labelwidth}
      \addtolength{\leftmargin}{\labelsep}%
   }%
}%
{\end{list}}



\numberwithin{equation}{section}

\newtheorem{thm}{Theorem}[section]
\newtheorem{lm}[thm]{Lemma}
\newtheorem{cor}[thm]{Corollary}
\newtheorem{prop}[thm]{Proposition}

\theoremstyle{remark}
\newtheorem{rem}[thm]{Remark}
\newtheorem*{rem*}{Remark}

\title[Bergman Similarity]{Similarity of operators in the Bergman Space Setting}

\author[Douglas]{Ronald G. Douglas}
\address[Ronald G. Douglas]{Department of Mathematics, Texas A\&M University, College
Station, TX, 77843, USA} \email{rdouglas@math.tamu.edu}

\author[Kwon]{Hyun-Kyoung Kwon}
\address[Hyun-Kyoung Kwon]{Department of Mathematical Sciences, Seoul National University, Seoul, 151-747, Republic of Korea}
\email{hyunkwon@snu.ac.kr}

\author[Treil]{Sergei Treil}
\address[Sergei Treil]{Department of Mathematics, Brown
University, Providence, RI, 02912, USA}
\email{treil@math.brown.edu}

\thanks{The work of R.~G.~Douglas was partially supported by a grant from the National Science Foundation.
The work of H.~Kwon was supported by the Basic Science Research
Program through the National Research Foundation of Korea (NRF)
funded by the Ministry of Education, Science, and Technology
(2011-0026989) and by the T.~J.~ Park Postdoctoral Fellowship. The
work of S.~Treil was supported by the National Science Foundation
under Grant DMS-0800876.}



\begin{abstract}
We give a necessary and sufficient condition for an $n$-
hypercontraction to be similar to the backward shift operator in a
weighted Bergman space. This characterization serves as a
generalization of the description given in the Hardy space
setting, where the geometry of the eigenvector bundles of the
operators is used.
\end{abstract}

\maketitle \setcounter{tocdepth}{1} \tableofcontents
\section*{Notation}
\begin{entry}
\item[$:=$] equal by definition;\medskip \item[$\C$] the complex
plane;\medskip \item[$\D$] the unit disk,
$\D:=\{z\in\C:\abs{z}<1\}$;\medskip \item[$\T$] the unit circle,
$\T:=\p\D=\{z\in\C:\abs{z}=1\}$;\medskip \item[$\frac{\p}{\p z},
\frac{\p}{\p \overline z}$] $\p$ and $\db$ derivatives:
$\frac{\p}{\p z}  := (\frac{\p}{\p x} - i \frac{\p}{\p y})/2$,
$\frac{\p}{\p \overline z}
 := (\frac{\p}{\p x} + i \frac{\p}{\p y})/2$; \medskip
\item[$\Delta$]normalized Laplacian, $\Delta := \db \p = \p \db =
\frac14\left(\frac{\p^2}{\p x^2} + \frac{\p^2}{\p y^2}
\right)$;\medskip \item[$\mathfrak{S}_2$]Hilbert-Schmidt  class of
operators;\medskip \item[$\norm{\cdot}, \nm\cdot \nm$]  norm:
since we are dealing with matrix- and operator-valued functions,
we will use the symbol $\|\,.\,\|$ (usually with a subscript) for
the norm in a function space, while $\nm\,.\,\nm$ is used for the
norm in the underlying vector (operator) space.
 Thus, for a vector-valued function $f$ the symbol $\|f\|_2$ denotes its $L^2$-norm, but the symbol $\nm f\nm$ stands
for the scalar-valued function whose value at a point $z$ is the
norm of the vector $f(z)$;  \medskip

\item[$H^\infty$] the space of all functions bounded and analytic
in $\D$; \medskip

\item[$L^\infty_{\!E_*\shto E}$] class of bounded functions on the
unit circle $\T$ whose values are bounded operators from a Hilbert
space $E_*$ to another one $E$ (the spaces $E$ and $E_*$ are not
supposed to be related in any way); \medskip

\item[$H^\infty_{\!E_*\shto E}$] operator Hardy class of bounded
analytic functions whose values are bounded  operators from $E_*$
to $E$:
$$
\|F\|_\infty := \sup_{z\in \D} \nm F(z)\nm=\underset{\xi\in
\T}{\operatorname{esssup}}\nm F(\xi)\nm; \medskip
$$

\item[$T_\Phi$] Toeplitz operator with symbol $\Phi$.

\medskip

\end{entry}

All Hilbert spaces are assumed to be separable. We also assume
that in a Hilbert space, an orthonormal basis is fixed so that any
operator $A:E\to E_*$ can be identified with its matrix. Thus,
besides the usual involution $A\mapsto A^*$ ($A^*$ is the adjoint
of $A$), we have two more: $A\mapsto A^T$ (transpose of the
matrix) and $A\mapsto \overline A$ (complex conjugation of the
matrix), so $A^* =(\overline A)^T =\overline{A^T}$. Although
everything in the paper can be presented in an invariant,
``coordinate-free'' form, the use of the transposition and complex
conjugation makes the notation simpler and more transparent.

\setcounter{section}{-1}

\section{Introduction}

We consider the question of when operators with a complete
analytic family of eigenvectors are similar. Recall that operators
$T_1$ and $T_2$ are said to be \emph{similar} if there exists a
bounded, invertible operator $A$ satisfying the intertwining
relation $AT_1 = T_2A$.

The problem of determining when two such operators are unitarily
equivalent goes back to the 1970's when the Cowen-Douglas class
was introduced in \cite{CowenDouglas}. It is proven there that
unitary equivalence has to do with the curvatures of the
eigenvector bundles of the operators and the partial derivatives
of them up to a certain order matching up. Unlike the unitary
equivalence case, however, the similarity problem posed a more
complicated situation (only some necessary conditions are listed
in \cite{CowenDouglas}) and no such criterion was obtained.

By adding the assumption that the operators in consideration be
contractive ($\|T \| \leq 1$), the authors in \cite{KwonTreil}
dealt with a special case of the problem; they gave a description
of operators with a complete analytic family of eigenvectors that
are similar to $S^*$, the backward shift operator on the Hardy
space $H^2$ (both scalar- and vector-valued) of the unit disk
$\D$. The backward shift $S^*$ is defined to be the adjoint of the
forward shift $S$,
$$Sf(z)=zf(z),$$ for $f \in H^2$, and similarity is shown to be equivalent
to the existence of a bounded (subharmonic) solution $\vf$ defined
on $\D$ to the Poisson equation
$$\Delta \vf=g,$$ where $g$ is a function related to the curvatures of the
eigenvector bundles of the operators.

One can ask whether the above characterization also holds for the
backward shift operators $B^*_\alpha$ defined on the weighted
Bergman spaces $A^2_{\alpha}$ (again, both scalar- and
vector-valued) of $\D$. If we let $P_\alpha$ denote the Bergman
projection and let $T_\Phi$ be the Toeplitz operator with symbol
$\Phi$ given by
$$T_\Phi f=P_{\alpha}(\Phi f),$$ then it
is easily seen that our backward shifts can be represented for $f
\in A^2_{\alpha}$ as
$$B^*_\alpha f(z)=P_{\alpha}(\bar{z}f(z))=T_{\bar{z}}f(z),$$
just like in the Hardy space case where the Bergman projections
are replaced by the Szeg\"{o} projection. We show in this paper
that the function-theoretic proof provided in \cite{KwonTreil} for
$S^*$ on $H^2$ can be applied to $B^*_\alpha$ on $A^2_\alpha$,
giving a generalization of the results there. Finally we mention
the recent paper \cite{DKKS}, where the authors use a Hilbert
module approach to prove that the similarity to the backward shift
operator on certain reproducing kernel Hilbert spaces can be
reduced to the similarity to $S^*$ on $H^2$.

\section{Preliminaries}
Let $n$ be a positive integer. Following the notation of
\cite{Agler2}, we denote by $\mathcal{M}_n$ the Hilbert space of
analytic functions on the unit disk $\D$ satisfying
$$
\|f\|^2_n:=\sum_{i=0}^{\infty} |\hat{f}(i)|^2 \frac{1}{ {n+i-1
\choose i}} < \infty,
$$
for $\mathcal{M}_n \ni f= \sum_{i=0}^{\infty} \hat{f}(i)z^i$. Note
that $\mathcal{M}_n$ corresponds to the Hardy space $H^2$ for
$n=1$, and for each positive integer $n \geq 2$, to the weighted
Bergman space $A^2_{n-2}$ defined by
$$
A^2_{n-2}=\{f \in \text{Hol}(\D): (n-1) \int_{\D} |f(z)|^2
(1-|z|^2)^{n-2}dA(z) < \infty\},
$$
for $dA$ the normalized area measure on $\D$. We can define the
vector-valued spaces $\mathcal{M}_{n,E}$ taking values in a
separable Hilbert space $E$ in a similar way.

On the space $\mathcal{M}_{n,E}$ are the forward shift operator
$S_{n,E}$, $S_{n,E}f(z)=zf(z)$ and the backward shift operator
$S^*_{n,E}$, its adjoint. Since $\mathcal{M}_n$ is a reproducing
kernel Hilbert space with reproducing kernel
$k_{\lambda}^n:=(1-\bar{\lambda}z)^{-n}, \la \in \mathbb{D}$, the
eigenvectors of $S^*_{n,E}$ corresponding to the eigenvalue
$\lambda$ is $k_{\bar{\lambda}}^ne$ for $e \in E$.

We now come to the definition of an $n$-hypercontraction
introduced in \cite{Agler} and \cite{Agler2}. Let $H$ be a Hilbert
space. An operator $T \in \mathcal{L}(H)$ is called an
\emph{n-hypercontraction} if
$$
\sum^k_{i=0} (-1)^i {{k}\choose{i}}T^{*i}T^i \geq 0,
$$
for all $1 \leq k \leq n$. Note that the 1-hypercontraction case
corresponds to the definition of the usual contraction.

Lastly, we recall the definition of a Carleson measure. Let
$$
Q(I):=\{z \in \mathbb{T}: \frac{z}{|z|} \in I, 1-|z| \leq |I|\},
$$
for $I \subseteq \mathbb{T}$, an arc of length $|I|$. A complex
measure $\mu$ in the closed unit disk is called a \emph{Carleson
measure} if for some constant $C$,
$$
|\mu|Q(I) \leq C |I|,
$$
where $|\mu|$ denotes the variation of $\mu$ \cite{Nik-book-v1}.

\section{Main results}
Let $n$ be a positive integer and $H$ a Hilbert space. We assume
the following for the operator $T \in \mathcal{L}(H)$ that we
consider:
\begin{enumerate}

\item $T$ is an $n-$hypercontraction;

\item $\spn\{\ker(T-\lambda): \lambda \in \D\}=H$; and

\item $\ker(T-\lambda)$ depend analytically on the spectral
parameter $\la \in \D$.\
\end{enumerate}

Assumption \cond3 says that for each $\lambda \in \mathbb{D}$, a
neighborhood $U_\lambda$ of $\lambda$ and an operator-valued
analytic function $F_\lambda$ defined on $U_\lambda$ that is
left-invertible in $L^\infty$ satisfying
$$\mbox{ran~}F_\lambda(w)=\ker(T-w),$$
for all $w \in U_\lambda$ exist. Therefore, the disjoint union
$\coprod_{\la\in\D} \ker(T-\la) =\{(\la, v_\la) : \la\in \D, v_\la
\in \ker(T-\la)\}$ is a hermitian, holomorphic vector bundle over
$\D$ with the metric inherited from $H$ and the natural projection
$\pi$, $\pi (\la, v_\la) =\la$. Note that assumption \cond3 then
implies that $\dim \ker (T-\lambda)$ is constant for all $\lambda
\in \D$. According to \cite{CowenDouglas}, the operators that
belong to the Cowen-Douglas class $B_m(\mathbb{D})$, or more
generally those with a certain Fredholm condition, for instance,
satisfy assumption \cond3.

We next mention that a \emph{bundle map} is a holomorphic map
between two holomorphic vector bundles over $\mathbb{D}$ that
linearly maps each fiber $\pi^{-1}(\lambda)$ of one bundle to the
corresponding fiber of the other bundle.

Now we state the main results of the paper:

\begin{thm} \label{t0.1}  Let $T \in \mathcal{L}(H)$ satisfy the above 3
assumptions with $\dim \ker(T-\la)=m < \infty$ for every $\la \in
\D$. Denote by $\Pi:\mathbb{D} \rightarrow \mathcal{L}(H)$ the
projection-valued function that assigns to each $\lambda \in
\mathbb{D}$, the orthogonal projection onto $\ker(T-\lambda)$. The
following statements are equivalent:

\begin{enumerate}
\item T is similar to the backward shift operator
$S^*_{n,\mathbb{C}^m}$ on $\mathcal{M}_{n,\mathbb{C}^m}$ via an
invertible operator $A:\mathcal{M}_{n,\mathbb{C}^m} \rightarrow
H$;

\item There exists a holomorphic bundle map bijection $\Psi$ from
the eigenvector bundle of $S^*_{n,\mathbb{C}^m}$ to that of $T$
such that for some constant $c>0$,
$$
\frac{1}{c} \| v_{\lambda}  \|_{\mathcal{M}_{n,\mathbb{C}^m}} \leq
\| \Psi ( v_{\lambda} )\| _{H} \leq c \| v_{\lambda} \|
_{\mathcal{M}_{n,\mathbb{C}^m}},
$$
for all $v_{\lambda} \in \ker(S^*_{n,\mathbb{C}^m}-\lambda)$ and
for all $\la\in \D$;

\item There exists a bounded solution $\vf$ defined on
$\mathbb{D}$ to the Poisson equation
$$
\Delta \vf (z) = \left \bracevert \frac{\partial\Pi(z)}{\partial
z}\right \bracevert ^2_{\mathfrak{S}_2}-\frac{mn}{(1-|z|^2)^2}.
$$
\end{enumerate}

\end{thm}

\begin{cor}
A contraction $T$ that satisfies assumptions \cond2, \cond3, and
$$
\sum^n_{i=0} (-1)^i {{n}\choose{i}}T^{*i}T^i \geq 0,
$$
enjoys the similarity characterization given in Theorem 2.1.
\end{cor}

\begin{cor}
A subnormal contraction that satisfies assumptions \cond2 and
\cond3 enjoys the similarity characterization given in Theorem
2.1.
\end{cor}

\begin{rem}
Note that the function $\Pi$ is $C^\infty$ and even real analytic
in the operator norm topology, so it does make sense to consider
$\frac{\partial\Pi(z)}{\partial z}$.
\end{rem}

\begin{rem}
Since $\left \bracevert \frac{\partial\Pi(z)}{\partial z}\right
\bracevert ^2_{\mathfrak{S}_2}-\frac{mn}{(1-|z|^2)^2} \geq 0$ (see
Section 3), $\vf$ is actually subharmonic.
\end{rem}

\begin{rem}
\label{rem0.3} For $m=1$, $-\left \bracevert
\frac{\partial\Pi(z)}{\partial z}\right \bracevert
^2_{\mathfrak{S}_2}$ and $-\frac{n}{(1-|z|^2)^2}$ represent the
curvatures of the eigenvector bundles of $T$ and of
$S^*_{n,\mathbb{C}}$, respectively (\cite{CowenDouglas},
\cite{GriffithsHarris}).
\end{rem}

\begin{rem}
\label{rem0.4} The existence of a bounded subharmonic function
$\vf$ defined on $\mathbb{D}$ satisfying
$$
\Delta \vf (z) \geq \left \bracevert
\frac{\partial\Pi(z)}{\partial z}\right \bracevert
^2_{\mathfrak{S}_2}-\frac{mn}{(1-|z|^2)^2}
$$
is equivalent to the uniform boundedness of the Green potential
$$
\mathcal G(\la) := \frac2\pi\iint_\D \log \left|
\frac{z-\la}{1-\overline \la z}\right| \left (\left \bracevert
\frac{\partial\Pi(z)}{\partial z} \right \bracevert
^2_{\mathfrak{S}_2}-\frac{mn}{(1-|z|^2)^2} \right)dxdy
$$
inside the unit disk $\D$.
\end{rem}

In order to prove Theorem 2.1, we first need to obtain a tensor
product structure for the operator $T$. Then since the equivalence
of statements (1) and (2) of Theorem 2.1 is obvious, and (3)
follows from the two statements

(4) The measure
$$
\left(\left \bracevert \frac{\partial\Pi(z)}{\partial z} \right
\bracevert
^2_{\mathfrak{S}_2}-\frac{mn}{(1-|z|^2)^2}\right)(1-|z|)dxdy
$$
is Carleson; and

(5) We have the estimate
$$
\left(\left\bracevert \frac{\partial\Pi(z)}{\partial z} \right
\bracevert
^2_{\mathfrak{S}_2}-\frac{mn}{(1-|z|^2)^2}\right)^{\frac{1}{2}}
\le \frac{C}{1-|z|},
$$
it suffices to show that (2) implies both (4) and (5) (Section 4)
and that (3) implies (1) (Section 5).

\section{Tensor structure of the eigenvector bundle}

\subsection{Structure of the eigenvector bundle of $T$}
 \label{s1.1}
The following theorem by J. Agler (\cite{Agler2}) proven through
the Rovnyak-de Branges construction is the first step to obtaining
a tensor product representation of the eigenvector bundle of $T$.
The reader is advised to consult \cite{Agler} also for an
alternative proof of the theorem based on complete positivity:

\begin{thm} Let $T \in \mathcal{L}(H)$. There exists a Hilbert space $E$ and an
$S^*_{n, E}$-invariant subspace $\mathcal{N} \subseteq
\mathcal{M}_{n, E}$ such that $T$ is unitarily equivalent to
$S^*_{n, E}|\mathcal{N}$ if and only if $T$ is an
$n$-hypercontraction with $\lim_k \|T^k h\|=0$ for all $h \in H$.
\end{thm}

Let us first observe that $\lim_k \|T^k h\|=0$ for $h \in H$ that
is a linear combination of the eigenvectors of $T$. According to
assumption \cond2, these linear combinations form a dense subspace
of $H$. Moreover, since an $n$-hypercontraction is automatically a
contraction, we have $\|T^k\| \leq 1$. We can thus employ a
standard argument to show that $\lim_k \|T^k h\|=0$ for all $h \in
H$.

Hence, the eigenspaces of $T=S^*_{n, E}|\mathcal{N}$ are given by
$$
\ker(T-\la ) =\{ k_{\overline\la}^n e: e\in \mathcal{N}(\la)\},
$$
where $k_{\la}^n=(1-\bar{\la}z)^{-n}$, $\la \in \mathbb{D}$, is
the reproducing kernel for $\mathcal{M}_n$ and
$\mathcal{N}(\la):=\{e \in E; k_{\bar{\la}}^n e \in
\mathcal{N}\}$. Note that by assumption \cond 3, the subspaces
$\mathcal{N}(\la)$ also depend analytically on the spectral
parameter $\la$, i.e., the family of subspaces $\mathcal{N}(\la)$
is a holomorphic vector bundle over $\D$.

Now, since the vector-valued Hilbert space $\mathcal{M}_{n, E}$
can be identified with $\mathcal{M}_n \otimes E$, the tensor
product of the Hilbert spaces $\mathcal{M}_n$ and $E$, the
eigenvector bundle of $T$ takes on the form
$$
\ker(T-\la) = \spn\{k_{\overline\la}^n\} \otimes \mathcal{N}(\la).
$$

\subsection{Calculation involving the eigenvector bundle of $T$} Recall that $\Pi(\la)$ stands for the orthogonal projection onto $\ker (T-\la)$. Using the
tensor structure given above, we can express $\Pi(\la)$ as
\begin{equation}
\label{Pi-tensor} \Pi(\la)=\Pi_1(\la)\otimes\Pi_2(\la),
\end{equation}
where $\Pi_1(\la)$ is the orthogonal projection from the space
$\mathcal{M}_n$ onto $\spn\{k_{\bar\la}^n\}$, and $\Pi_2(\la)$ is
the orthogonal projection from $E$ onto $\mathcal{N}(\la)$. We
remark that $\rank \Pi(\la)=\rank \Pi_2(\la)=m$.

\begin{lm}
\label{l-PdP}For $\la \in \D$, let $\Gamma(\lambda)$ be orthogonal
projections onto an analytic family of subspaces (holomorphic
vector bundle). Then the identities
$$\Gamma(z) \frac{\p
\Gamma(z)}{\p z} =0$$ and
$$
(I-\Gamma(z))\frac{\p\Gamma(z)}{\p z} \Gamma(z) =
\frac{\p\Gamma(z)}{\p z}
$$
hold.
\end{lm}

\begin{proof}[Proof of Lemma \ref{l-PdP}]
Since the family of subspaces is a holomorphic vector bundle, it
can be locally expressed as $\ran F(\la)$, where $F$ is an
analytic, left-invertible operator-valued function. Thus,
$\Gamma=F(F^*F)^{-1}F^*$. We obtain through direct computation
that
$$
\frac{\partial\Gamma(z)}{\partial
z}=(I-\Gamma(z))F'(z)(F(z)^*F(z))^{-1}F(z)^*.
$$
Since $\Gamma(z)$ is a projection, we immediately arrive at the
first identity. For the second one, we note that
$\Gamma(z)F(z)=F(z)$ implies $\frac{\partial\Gamma(z)}{\partial
z}\Gamma(z)=\frac{\partial\Gamma(z)}{\partial z}$. We then invoke
the first identity.
\end{proof}

\begin{lm}
\label{curv-shift} The projection $\Pi_1(\la)$ satisfies the
identity
$$
\left \bracevert\frac{\partial\Pi_1(z)}{\partial z} \right
\bracevert^2_{\mathfrak{S}_2} =n{(1-|z|^2)^{-2}}.
$$
\end{lm}

\begin{proof}[Proof of Lemma \ref{curv-shift}]
We first use the reproducing kernel property of
$k_\la^n=1/{(1-\bar{\lambda} z)^n}$ to see that $\|k_\la^n\|_2^2 =
\La k_\la^n, k_\la^n\Ra = (1-|\la|^2)^{-n}$. Thus
$$
\Pi_1(\la) f =\|k_{\bar\la}^n\|_2^{-2} \La f, k_{\bar\la}^n \Ra
k_{\bar\la}^n = (1-|\la|^2)^n f(\bar\la) k_{\bar\la}^n,
$$
for $f \in M_n$. We next use the fact that $\frac{\p
f(\bar\la)}{\p\la}=0$ and $\frac{\p}{\p\la} k_{\bar\la}^n(z) =
\frac{nz}{(1-\la z)^{n+1}}=:\tilde{k}_{\bar{\lambda}}^n(z)$ to get
\begin{equation}
\label{PdP-1.1} \frac{\p\Pi_1(\la)}{\p\la} f =
(1-|\la|^2)^{n-1}f(\bar\la) \left( -n\bar\la k_{\bar\la}^n +
(1-|\la|^2) \wt k_{\bar\la}^n \right).
\end{equation}

Since $\La f, \widetilde k_{\la}^n\Ra = f'(\la)$ for $f \in M_n$,
$$
\|\wt k_{\la}^n \|_2^2 =
\frac{n(1+n|\lambda|^2)}{(1-|\lambda|^2)^{n+2}} = \|\wt
k_{\bar\la}^n \|_2^2.
$$
Once again, the reproducing property of $k_\la^n$ implies that
$$
\La\wt k_{\bar\la}^n , k_{\bar\la}^n\Ra =
\frac{n\bar\la}{(1-|\la|^2)^{n+1}}.
$$
Taking all these calculations into account, we conclude that
$$
\|-n\bar\la k_{\bar\la}^n + (1-|\la|^2) \wt k_{\bar\la}^n\|_2^2 =
n(1-|\la|^2)^{-n}.
$$
Thus,
$$
\left \bracevert\frac{\partial\Pi_1(\lambda)}{\partial \lambda}
\right \bracevert^2 = n(1-|\lambda|^2)^{-2},
$$
and we note from \eqref{PdP-1.1} that
$$
\rank \frac{\partial
\Pi_1(\la)}{\partial \la}=1.
$$
Therefore,
$$\left \bracevert\frac{\partial\Pi_1(\la)}{\partial \la} \right
\bracevert^2_{\mathfrak{S}_2} = \left
\bracevert\frac{\partial\Pi_1(\la)}{\partial \la} \right
\bracevert^2 = n{(1-|\la|^2)^{-2}}.$$

\end{proof}

\begin{lm}
\label{l-curv-T} The projection $\Pi(\la)$ satisfies the identity
\begin{align*}
{\left \bracevert \frac{\partial\Pi(z)}{\partial z}\right
\bracevert}^2_{\mathfrak{S}_2} & = m{\left \bracevert
\frac{\partial\Pi_1(z)}{\partial z}\right
\bracevert}^2_{\mathfrak{S}_2} + {\left \bracevert
\frac{\partial\Pi_2(z)}{\partial z}\right
\bracevert}^2_{\mathfrak{S}_2}
\\
& = \frac{mn}{(1-|z|^2)^2} \, + {\left \bracevert
\frac{\partial\Pi_2(z)}{\partial z}\right
\bracevert}^2_{\mathfrak{S}_2}.
\end{align*}
\end{lm}

\begin{proof}[Proof of Lemma \ref{l-curv-T}]

We apply the product rule to \eqref{Pi-tensor} to obtain
$$
\frac{\partial\Pi(\la)}{\partial \la} =\frac{\partial\Pi_1(\la
)}{\partial \la } \otimes\Pi_2(\la )+\Pi_1(\la
)\otimes\frac{\partial\Pi_2(\la )}{\partial \la } =: X + Y.
$$
Since $\Pi_2(\la) \frac{\p \Pi_2(\la)}{\p \la} = 0$ by Lemma
\ref{l-PdP}, $X^*Y=0$. Therefore,
$$
\left\bracevert X+Y\right\bracevert_{\mathfrak{S}_2}^2 = \tr X^*X
+ \tr Y^*Y + 2\re\tr (X^*Y) = \left\bracevert
X\right\bracevert_{\mathfrak{S}_2}^2 + \left\bracevert
Y\right\bracevert_{\mathfrak{S}_2}^2.
$$
Using the fact that $\left\bracevert A \otimes B
\right\bracevert^2_{\mathfrak{S}_2}=\left\bracevert A
\right\bracevert^2_{\mathfrak{S}_2}\left \bracevert B\right
\bracevert^2_{\mathfrak{S}_2}$ and that $\left\bracevert
P\right\bracevert_{\mathfrak{S}_2}^2 = \rank P$ for an orthogonal
projection $P$, we get
$$
\left\bracevert \frac{\partial\Pi(\la)}{\partial \la}
\right\bracevert_{\mathfrak{S}_2}^2 = m \left\bracevert
\frac{\partial\Pi_1(\la)}{\partial \la}
\right\bracevert_{\mathfrak{S}_2}^2 + \left\bracevert
\frac{\partial\Pi_2(\la)}{\partial \la}
\right\bracevert_{\mathfrak{S}_2}^2\,.
$$
The result now follows from Lemma \ref{curv-shift}.
 \end{proof}

\section{Proof of ``(2) implies (3)"}

Let us mention again that statements \cond 4 and \cond 5 of
Section 2 together imply statement \cond 3 of Theorem 2.1.
Moreover, since we have by Lemma \ref{l-curv-T}
$$
{\left \bracevert
\frac{\partial\Pi_2(\lambda)}{\partial\lambda}\right
\bracevert}^2_{\mathfrak{S}_2} = {\left \bracevert
\frac{\partial\Pi(\lambda)}{\partial\lambda}\right
\bracevert}^2_{\mathfrak{S}_2} - \frac{mn}{(1-|\la|^2)^2} ,\,
$$
the quantity ${\left \bracevert
\frac{\partial\Pi(\lambda)}{\partial\lambda}\right
\bracevert}^2_{\mathfrak{S}_2} - \frac{mn}{(1-|\la|^2)^2}$ in
statements \cond4 and \cond5 can be replaced by ${\left \bracevert
\frac{\partial\Pi_2(\lambda)}{\partial\lambda}\right
\bracevert}^2_{\mathfrak{S}_2}$.

Assume that statement \cond2 of Theorem 2.1 holds to guarantee the
existence of a holomorphic bundle map bijection $\Psi$ with a
certain property between the eigenvector bundles. Then for all $e
\in \mathbb{C}^m$,
$$
\Psi(k_{\bar{\lambda}}^ne)=k_{\bar{\lambda}}^n \cdot F(\lambda)e,
$$
where $F$ is some function in $H^{\infty}_{\mathbb{C}^m
\rightarrow E}$ satisfying $ \ran F(\lambda)= \mathcal{N}(\la)$
and $ \label{un_eq} c^{-1} I \le F^*F\le c I $. Thus it makes
sense to consider $(F^*F)^{-1}$ and we can express the orthogonal
projection $\Pi_2(\la)$ from $E$ onto $\mathcal{N}(\la)$ in terms
of $F$ as
$$
\Pi_2=F(F^*F)^{-1}F^*.
$$
Since $\frac{\partial\Pi_2(z)}{\partial
z}=(I-\Pi_2(z))F'(z)(F(z)^*F(z))^{-1}F(z)^*$, we get
\begin{equation}
\label{2.2} \left \bracevert \frac{\partial\Pi_2(z)}{\partial z}
\right \bracevert \leq C\nm F'(z) \nm.
\end{equation}

Lastly, we note that since $F$ is a bounded analytic function
taking values in a Hilbert space, the estimate
\begin{equation}
 \label{f'1} \nm F'(z) \nm \leq {C}/{(1-|z|)}
\end{equation}
holds, and the measure
\begin{equation}
\label{f'2}
 \nm F'(z) \nm ^2(1-|z|)dxdy
\end{equation} is Carleson. The first estimate \eqref{f'1} is well-known for
scalar-valued analytic functions, and one can pick $x^*=x^*(z)$,
$\nm x^*\nm=1$ in the dual space $X^*$ such that $\La F'(z),
x^*\Ra = \nm F'(z)\nm$ to show that it holds for functions with
values in a Banach space $X$. To see that the Carleson measure
condition \eqref{f'2} holds, we use \emph{Uchiyama's Lemma} which
states that for a bounded subharmonic function $u$, the measure
$\Delta u(z) (1-|z|) dxdy$ is Carleson. We apply this Lemma to the
function $u(z) = \nm F(z)\nm^2$ and note that $\Delta \nm
F(z)\nm^2 = \nm F'(z)\nm^2$. By \eqref{2.2}, \eqref{f'1}, and
\eqref{f'2}, we get the existence of a bounded subharmonic
function $\vf$ on $\D$ with
$$
\Delta \vf (z) \geq \left \bracevert
\frac{\partial\Pi_2(z)}{\partial z}\right \bracevert
^2_{\mathfrak{S}_2}.
$$

To obtain equality, we note that the equation $\Delta u(z)= f(z)$
always has a solution, namely, the Green potential
$$
\mathcal G_f(\la) := \frac2\pi\iint_\D \log \left|
\frac{z-\la}{1-\overline \la z}\right| f(z) dxdy.
$$
But since
$$
G_{\Delta \vf} \leq G_{\left \bracevert
\frac{\partial\Pi_2}{\partial z}\right \bracevert
^2_{\mathfrak{S}_2}} \leq 0,
$$
and $G_{\Delta \vf}$ is bounded, the subharmonic solution
$G_{\left \bracevert \frac{\partial\Pi_2}{\partial z}\right
\bracevert ^2_{\mathfrak{S}_2}}$ to
$$ \Delta u(z) = \left \bracevert \frac{\partial\Pi_2(z)}{\partial
z}\right \bracevert ^2_{\mathfrak{S}_2}
$$
is bounded as well.

\section{Proof of ``(3) implies (1)"}
The goal of this section is to prove the existence of a bounded,
invertible operator $A:\mathcal{M}_{n, \mathbb{C}^m} \to
\mathcal{N}$ such that $AS^*_{n, \mathbb{C}^m}=(S^*_{n,
E}|\mathcal{N})A$. We first consider the following theorem that
will let us get a bounded, analytic projection onto $\ran
\mathcal{N}(z)$ for $z \in \mathbb{D}$ \cite{TreilWick}.

\begin{thm}
\label{t_Tr-W} Let $\Gamma:\D \to \mathcal{L}(H)$ be a $\cC^2$
function whose values are orthogonal projections in $H$. Assume
that $\Gamma$ satisfies the identity $\Gamma(z)
\frac{\p\Gamma(z)}{\p z}=0$ for all $z \in \mathbb{D}$. Given a
bounded, subharmonic function $\vf$ with
$$
\Delta \vf (z) \ge \left \bracevert \frac{\p\Gamma(z)}{\p z}\right
\bracevert^2\qquad \text{for all }z\in \D,
$$
there exists a bounded analytic projection onto $\Gamma (z)$,
i.e., a function $\cP\in H^\infty_{H \shto H}$ such that $\cP(z)$
is a projection onto $\ran\Gamma(z)$ for all $z\in \D$.
\end{thm}

We know from Lemma \ref{l-PdP} that the function $\Pi_2$ whose
values are orthogonal projections from $E$ onto $\mathcal{N}(\la)$
satisfies the identity $\Pi_2(z) \frac{\p\Pi_2(z)}{\p z}=0$ so
that the above theorem is applicable. We thus get a bounded,
analytic projection $\cP(z)$ onto $\ran \Pi_2(z)=\mathcal{N}(z)$,
and consider the inner-outer factorization $\cP = \cP\ti i \cP\ti
o$ of $\cP$, where $\cP\ti i\in H^\infty_{E_*\shto E}$ for some
Hilbert space $E_*$, is an inner function and $\cP\ti o\in
H^\infty_{E \shto E_*}$ is an outer function. We then define a
function $\cQ_i$ on $\D$ by
$$
\cQ_i (z) := \cP\ti i(\bar z),
$$
and form the anti-analytic Toeplitz operator $T_{\cQ_i}$.

We claim that this bounded Toeplitz operator $T_{\cQ_i}$ is an
invertible operator that establishes similarity. To this end, we
need to prove the following three statements:
\begin{enumerate}
\item $T_{\bar{z}}T_{\cQ_i}=T_{\cQ_i}T_{\bar{z}}$;

\item $T_{\cQ_i}$ is left-invertible; and

\item $\ran T_{\cQ_i}=\mathcal{N}$.

\end{enumerate}

We begin by recalling some well-known facts about Toeplitz
operators on the vector-valued spaces $\mathcal{M}_{n}$. Let $F,G
\in H^\infty_{E \rightarrow E_*}$:
\begin{equation}
T_{FG}=T_FT_G; \text{ and}
\end{equation}
\begin{equation}
T_{F^*}k_{\la}^ne=k_{\la}^nF^*(\la)e \text{ for }e \in E_*.
\end{equation}

Since $\cQ_i^* \in H^{\infty}_{E \rightarrow E_*}$, statement (1)
easily follows from (5.1). To prove (2), we consider the following
Lemma.

\begin{lm}
\label{l3.3} $\cP\ti o(z) \cP\ti i(z) \equiv I$ for all $z\in \D$.
\end{lm}

\begin{proof}[Proof of Lemma \ref{l3.3}]
By (5.1), we have that
$$
T_{\cP\ti i}  T_{\cP\ti o}  = T_\cP = T_{\cP^2} = T_{\cP\ti i
\cP\ti o \cP\ti i \cP\ti o} = T_{\cP\ti i} T_{\cP\ti o \cP\ti i}
T_{\cP\ti o} .
$$
Since $T_{\cP\ti o}$ has dense range and $\ker T_{\cP\ti i}
=\{0\}$, $T_{\cP\ti o \cP\ti i} = I$, so $\cP\ti o \cP\ti i \equiv
I$ for all $z\in \D$.
\end{proof}
We then note that since $\cQ^*_o \in H^\infty_{E_* \rightarrow
E}$, where $\cQ_o(z) :=\cP\ti o(\bar z)$, we can once again use
(5.1) to conclude that
$$
T_{\cQ_o}T_{\cQ_i}=T_{\cQ_ o \cQ_i}=I.
$$

It now remains to show statement (3). The inclusion
$\mathcal{N}(\la) = \ran \cP(\la) \subset \ran\cP\ti i(\la)$ is
obvious due to the factorization $\cP= \cP\ti i \cP\ti o$. For the
other inclusion, since $\ran \cP\ti o(\la)$ is dense in $E_*$ for
all $\la \in \D$ , and $\cP\ti i(\la) \ran\cP\ti o(\la) =
\mathcal{N}(\la)$, $\ran \cP\ti i(\la) \subset \mathcal{N}(\la)$.
Thus,
\begin{equation}
\ran\cP\ti i(\la) =\mathcal{N}(\la).
\end{equation}
 We next observe that by (5.2),

\begin{equation}
    T_{\cQ i} k_{\bar\la}^n e = k_{\bar\la}^n \cQ_i (\bar\la)e = k_{\bar\la}^n \cP\ti
    i(\la)e,
\end{equation}
for all $e \in E_*$. Then (3) follows from (5.4), the fact that
$\spn\{k_{\la}^n: \la \in \D\}=M_n$, and assumption \cond2 that
$\spn\{\ker(T-\lambda): \lambda \in \D\}=H$.
 \hfill \qed

\section{Proof of corollaries}
Now we prove the corollaries of Theorem 2.1 that appeared in
Section 2. The statements used in these proofs are contained in
\cite{Agler2}.
\begin{proof}[Proof of Corollary 2.2]
We have $\lim_k \|T^k h\|=0$ for $h \in H$ that is a linear
combination of the eigenvectors of $T$, which by assumption \cond2
is dense in $H$. If $T$ is a contraction, then $\|T^k\| \leq 1$,
so that $\lim_k \|T^k h\|=0$ for all $h \in H$. Now we use the
result that an operator $T \in \mathcal{L}(H)$ with
$$
\sum^n_{i=0} (-1)^i {{n}\choose{i}}T^{*i}T^i \geq 0,
$$
and such that $\lim_k \|T^k h\|=0$ for all $h \in H$ is an
$n$-hypercontraction.
\end{proof}

\begin{proof}[Proof of Corollary 2.3]
An operator $T$ is an $n$-hypercontraction for every $n$ if and
only if $\|T\| \leq 1$ and $T$ is subnormal \cite{Embry}.
\end{proof}

\def\cprime{$'$}
\providecommand{\bysame}{\leavevmode\hbox
to3em{\hrulefill}\thinspace}


\begin{thebibliography}{99}
\bibitem{Agler}J. Agler, \emph{The Arveson extension theorem and
coanalytic models, }{Integr. Equat. Op. Thy. } $\mathbf{5}$
(1982), 608-631.

\bibitem{Agler2}J. Agler, \emph{Hypercontractions and
subnormality, }{J. Operator Theory,} $\mathbf{13}$ (1985),
203-217.

\bibitem{Carleson}L. Carleson, \emph{Interpolations by bounded analytic functions and the corona problem,} {Ann. of Math.}
 $\mathbf{76}$ (1962), No.~3, 547-559.

\bibitem{CowenDouglas}M. J. Cowen and R. G. Douglas,
\emph{Complex geometry and operator theory,} {Acta. Math.
}$\mathbf{141}$ (1978), 187-261.

\bibitem{DKKS}
R. G. Douglas, Y. Kim, H. Kwon, and J. Sarkar, {\em Curvature
invariant and generalized canonical operator models - II},
preprint.

\bibitem{Embry}
M. Embry, \emph{A generalization of the Halmos-Bram criterion for
subnormality}, Acta Sci. Math. (Szeged), $\mathbf{35}$ (1973),
61-64.

\bibitem{GriffithsHarris}P. Griffiths and J. Harris,
\emph{Principles of Algebraic Geometry}, John Wiley $\&$ Sons,
Inc., New York, 1994.

\bibitem{KwonTreil}
H-K. Kwon and S. Treil, {\em Similarity of operators and geometry
of eigenvector bundles},  Publ. Mat. $\mathbf{53}$ (2009), No.~2,
417-438.

\bibitem{Nik-book-v1}
N.~K.~Nikolski, \emph{Operators, Functions, and Systems: An Easy
Reading.
  {V}ol. 1: Hardy, Hankel, and Toeplitz}, Mathematical Surveys and
Monographs,
  Vol.~92, American Mathematical Society, Providence, RI, 2002, Translated from
  the French by Andreas Hartmann.

\bibitem{Nik-shift}
\bysame, \emph{{Treatise on the Shift Operator}},
  Grundlehren der Mathematischen Wissenschaften [Fundamental Principles of
  Mathematical Sciences], vol. 273, Springer-Verlag, Berlin, 1986, Spectral
  function theory, With an appendix by S. V. Hru\v s\v cev [S. V.
  Khrushch{\"e}v] and V. V. Peller, Translated from the Russian by Jaak Peetre.

\bibitem{TreilWick} S. Treil and B. D. Wick, \emph{Analytic projections, corona problem and geometry of holomorphic vector
bundles,} {J. Amer. Math. Soc.}{ \bf 22} (2009), No.~1, 55--76.

 \end{thebibliography}
\end{document}